\documentclass[12pt]{elsarticle}

\usepackage{graphicx}

\usepackage{amsmath}
\usepackage{amssymb}
\usepackage[all]{xy}
\usepackage{enumerate}
\usepackage{array}
\usepackage{amsthm}
\usepackage{amscd}
\usepackage{color}
\usepackage{hyperref}
\hypersetup{
    colorlinks,
    citecolor=red,
    filecolor=black,
    linkcolor=blue,
    urlcolor=black
}
\hypersetup{linktocpage}
\usepackage{multirow}

\usepackage{mathtools}
\usepackage{endfloat}

\usepackage{centernot}

\newcommand \N{\mathbb{N}}
\newcommand \Z{\mathbb{Z}}
\newcommand \R{\mathbb{R}}
\newcommand \F{\mathcal{F}}
\newcommand \U{\mathcal{U}}
\newcommand{\Lip}{\mathrm{Lip}}
\newcommand{\diam}{\mathrm{diam}}

\newcommand{\lto}{\longrightarrow}

\newtheorem{prop}{Proposition}[section]
\newtheorem{teo}[prop]{Theorem}
\newtheorem{ct}[prop]{Counterexample}

\newtheorem{cor}[prop]{Corollary}

\newtheorem{defn}[prop]{Definition}

\theoremstyle{definition}
\newtheorem{ejem}{Example}

\theoremstyle{remark}

\begin{document}

\begin{frontmatter}



\title{Counterexamples for IFS-attractors}
\author{Magdalena Nowak\fnref{fn1}}
\ead{magdalena.nowak805@gmail.com}
\address{Jan Kochanowski University in Kielce, \'Swietokrzyska 15, 25-406 Kielce, POLAND}
\fntext[fn1]{The first author was partially supported by National Science Centre grant DEC-2012/07/N/ST1/03551}

\author{M. Fern\'andez-Mart\'{\i}nez\fnref{fn2}}
\ead{fmm124@gmail.com}
\address{University Centre of Defence at the Spanish Air Force Academy, MDE-UPCT,\\ 30720 Santiago de la Ribera, Murcia, SPAIN}
\fntext[fn2]{The second author specially acknowledges the valuable support
provided by Centro Universitario de la Defensa en la Academia General del
Aire de San Javier (Murcia, Spain).}





%
%
%

\begin{abstract}
In this paper, we deal with the part of Fractal Theory related to finite families of (weak) contractions, called iterated function systems (IFS, herein). An attractor is a compact set which remains invariant for such a family. Thus, we consider spaces homeomorphic to attractors of either IFS or weak IFS, as well, which we will refer to as Banach and topological fractals, respectively. We present a collection of counterexamples in order to show that all the presented definitions are essential, though they are not equivalent in general.
\end{abstract}


\end{frontmatter}
\section{Introduction}
Since the rise of Fractal Theory started with the first works of pioneer Mandelbrot \cite{MAN77,MAN82}, the relevance of fractals in a wide range of scientific areas has increased during the last years. It is worth mentioning, in this occassion, the connection among fractals and dynamical systems. Indeed, it is well known and established that fractal dimension provides a useful measure about the chaotic behavior of a given dynamical system. This bas been carried out classically through the box dimension (see, e.g., \cite[Subsection 8.4]{ROB95}), mainly due to the easiness when being calculated or estimated in empirical applications, though the Hausdorff dimension, which is the oldest and also the most accurate (by definition) model for fractal dimension has been considered, too, at least from a theoretical point of view. Recently, novel alternatives, with some desirable analytical properties as it happens with the Hausdorff dimension model, and being as easy to calculate as the box dimension, have been contributed. In particular, a new fractal dimension algorithm, specially appropriate to deal with time series, has allowed the authors to study the chaotic behavior of planar oscillations described by a satellite in the field of astrodynamics, based on a Beletsky model \cite{NDY15}. Moreover, also the Hurst exponent (resp. the self-similarity exponent) turns into another serious candidate to tackle this task, and has been already explored as a chaos measure for dynamical systems \cite{TAR15}. They provide alternative chaos measurements within the advantage to be less computationally expensive than classical Lyapunov exponent.

Another interesting topic that fractal dimension allows to deal with is about the study of self-similar sets, which become a special kind of fractal sets which could be equipped in a quite natural way by a fractal structure (see \cite[Definition 4.4]{MSG12}). To calculate the fractal dimension of this kind of fractals through a explicit formula constitutes a worth-mentioning task, and this has been explored previously through both the box-counting and the Hausdorff fractal dimensions, though a certain restrictive hypothesis is required (see \cite{MOR46}). Similar results in this line have been also contributed through novel definitions of fractal dimension specially developed for any fractal structure (see, e.g., \cite[Theorem 4.19]{DIM1}, \cite[Theorems 4.19 \& 4.20, and Corollary 4.22]{DIM3}). In particular, one of such models allows to calculate the fractal dimension for self-similar sets (equipped with their natural fractal structure) with the open set condition not having to be satisfied (see \cite[Definition 4.2]{DIM3}).

In this paper, two notions for a self-similar set are explored. Thus, a space being homeomorphic to any attractor of an iterated function systems is called a Banach fractal, whereas a topological fractal basically consists of a space which is homeomorphic to a weak iterated function system (namely, such an IFS is constructed through a finite set of weak contractive mappings). The main goal in the present work is to state that these concepts, though essential, are not equivalent in general. This is carried out through a collection of appropriately constructed counterexamples, as can be seen in forthcoming Subsections \ref{sub:1} \& \ref{sub:2}.

\section{Banach and topological fractals}
In this section, some concepts regarding (weak) IFS-attractors, as well as Banach and topological fractals, are firsly recalled and theoretically connected. Moreover, in Subsection \ref{sub:1}, three appropriately chosen counterexamples are provided in order to show that the reciprocal links are not true, in general. On the other hand, Subsection \ref{sub:2} allow the authors to point out that though no Peano continuum not being a topological fractal is still known, scattered spaces do provide a topological context where many examples of this could be found. This is illustrated through a standard construction involving the Cantor-Bendixson derivative.

\begin{defn}
For a metric space $(X,d)$, let us define an IFS as a finite family $\F$ of contractive self-maps (with all the Lipschitz constants, denoted by $\Lip$, being $<1$). The unique compact set $A\subset X$ which satisfies $A=\bigcup_{f\in\F}f(A)$ is called the attractor of the IFS $\F$ or IFS-attractor.
\end{defn}

It is a standard fact from Fractal Theory that there exists an attractor for every IFS on a complete metric space $X$  (see, e.g., \cite{Barnsley}).  It is a consequence of Banach contraction principle. Namely, the iterated function system $\F$ generates a contractive operator $\F(K)=\bigcup_{f\in \F}f(K)$ acting on a (complete) space of nonempty, compact subsets of $X$ with the Hausdorff metric.  

We can even consider a weaker version for the notion of IFS.

\begin{defn}
By weak IFS on a metric space $(X,d)$, we understand a finite family $\F$ of self-maps which are weak contractions, namely, it is satisfied that for every distinct points $x,y\in X$, and for all $f\in\F$, it holds that $d(f(x) ,f(y))<d(x,y)$. A compact set $A$ such that $A=\bigcup_{f\in\F}f(A)$ for such a family $\F$, will be named weak IFS-attractor. 
\end{defn}

It is worth mentioning that in \cite[Remark 7.1]{Edelstein}, it was proved that whether the space $X$ is compact, then there exists the attractor for every weak IFS on $X$. In this paper, though,
we will also consider spaces homeomorphic to attractors for (weak) IFS. These concepts are stated next.
 
\begin{defn}
A Banach fractal is a compact metrizable space which is homeomorphic to the attractor of an IFS on a complete metric space. 
\end{defn}

We also present the notion of topological fractal (attractor of the topological IFS), considered in \cite{BanakhNowak,mihail}, which is homeomorphic to the attractor of a weak IFS on a compact metric space. 

\begin{defn}
A Hausdorff topological space $X$ is called a topological fractal provided that $X=\bigcup_{f\in\F}f(X)$, where $\F$ is a finite family of continuous self-maps such that for every open cover $\U$ of the space $X$, there exists $n\in\N$ such that for any maps $f_1,\dots,f_n\in\F$, the set $f_1\circ\ldots\circ f_n(X)$ contains in some set $U\in\U$.
\end{defn}
Note that every compact metric space $X$ is a topological IFS-attractor whether for any of its open covers $\U$, the diameter of the set $f_1\circ\ldots\circ f_n(X)$ is less than the Lebesgue number of $\U$, for some $n\in\N$, and every $f_1,\dots,f_n\in\F$.  Recall that the Lebesgue number $\lambda$ for an open cover $\U$ of a compact metric space $X$, is a positive real number, such that every subset of $X$ whose diameter is less than $\lambda$, is contained in some element of $\U$.

Surprisingly, it holds that a space is a topological fractal if and only if it is metrizable and homeomorphic to the attractor of some weak IFS. This result was recently proved in \cite[Corollary 6.4]{BKNNS}.

Regarding the present work, next we will show that all the presented definitions, while essential, are not equivalent in general. Indeed, it is easy to check the connections displayed in the following diagram:
$$\xymatrix{
\mbox{IFS-attractor}\ar@{=>}[r]\ar@{=>}[d]&\mbox{weak IFS-attractor}\ar@{=>}[d]\\
\mbox{Banach fractal}\ar@{=>}[r]&\mbox{topological fractal}\\
}$$
though the inverse implications do not hold. We deal with in the upcoming subsections.

\subsection{Collection of counterexamples}\label{sub:1}
First of all, we introduce a version of the so-called \emph{harmonic spiral} space (see \cite{Sanders}), which allows us to show that neither weak IFS-attractor nor Banach fractal imply IFS-attractor, in general.

\begin{ct}
There exists an arc called ``the snake'', which is both a weak IFS-attractor and Banach fractal but not IFS-attractor. 
\end{ct}

\begin{proof}
Firstly, to provide a description for the snake, let us switch to standard polar coordinates $(r,\alpha)$ on $\R^2$, that is $(x,y)=(r\cos\alpha,r\sin\alpha)$. Thus, the snake consists of circular sectors
$$O_n = \Big\{\Big(\frac1n,\alpha\Big)\in\R^2\colon\alpha\in\Big(\frac{\pi}2,2\pi\Big)\Big\},$$
and intervals
$$I_n =\Bigg\{(r,\alpha)\in\R^2: r\in\Bigg[\frac1{n+1},\frac1n\Bigg], \alpha= n\bmod{2}\cdot\frac{\pi}{2}\Bigg\},$$
for $n\geq1$. Hence, the snake is defined as $S = \bigcup_{n=1}^\infty (O_n\cup I_n) \cup \{(0,0)\}$.

\begin{figure}[h]
\centering
\includegraphics[width=250px]{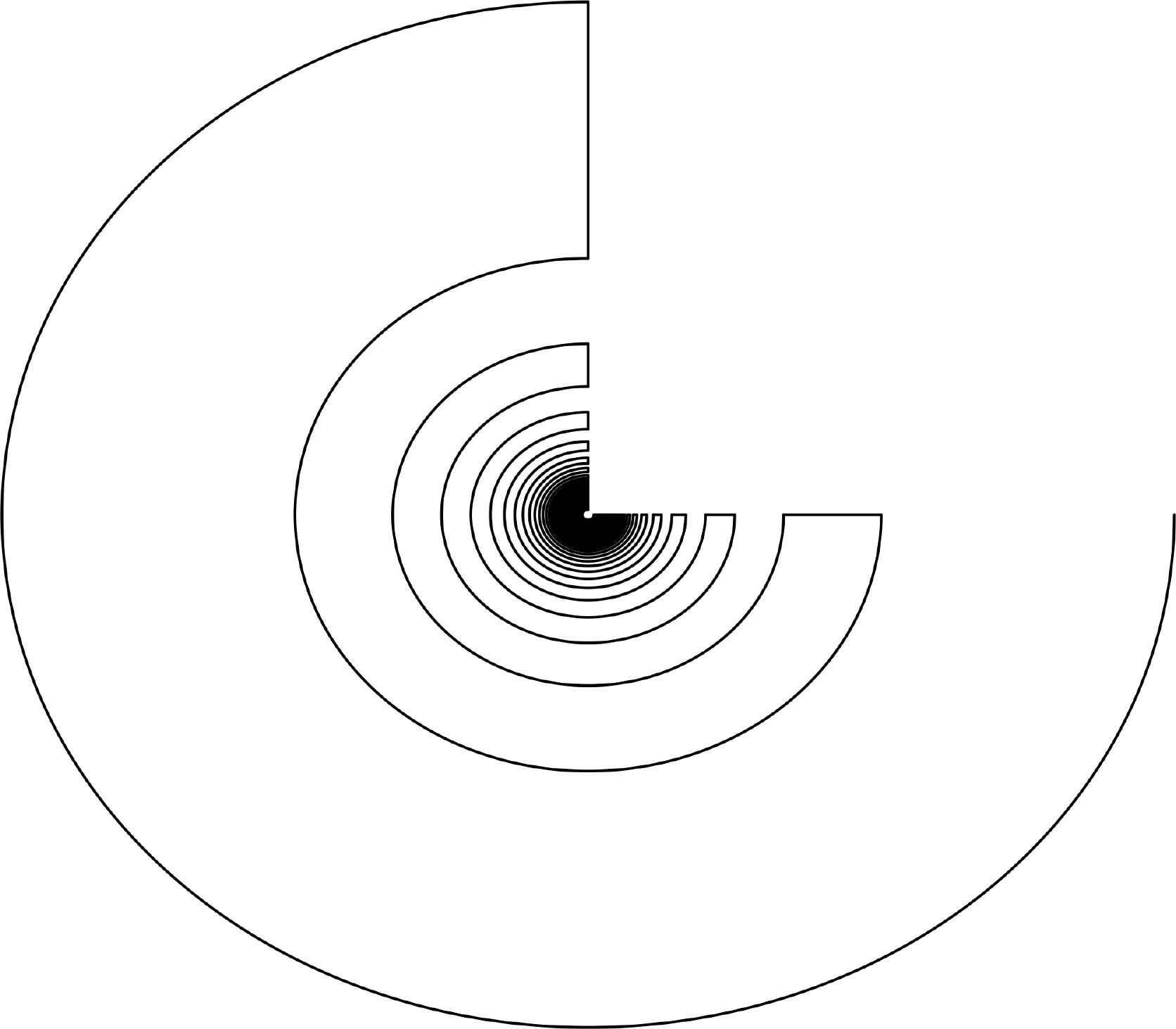}
\caption{The arc $S$, named \emph{the snake}.}
\end{figure}
This is an arc, namely, a space homeomorphic to the closed unit interval. Since $[0,1]$ is a standard IFS-attractor, the snake becomes a Banach fractal. On the other hand, $S$ has infinite length and satisfies the following theorem proved by Sanders in \cite[Theorem 4.1]{Sanders}: 

\begin{teo}\label{thm_Sanders}
An arc $A\subset\R^n$ is not an IFS-attractor if for one of its endpoints $a \in A$, the following conditions are satisfied:
\begin{itemize}
\item for all $x, y \in A\setminus\{a\}$, the length of the subarc of $A$ with endpoints $x$ and $y$ is finite, and
\item for every $x\in A\setminus\{a\}$, the length of the subarc of $A$ with endpoints $x$ and $a$ is infinite.
\end{itemize}
\end{teo}

Let us clarify the notions consisting of endpoints as well as the length of an arc. 
Indeed, take an arc $A=e([0,1])$, where $e\colon[0,1]\lto\R^n$ is an embedding. The {\em endpoints} of the arc $A=e([0,1])$, are the points $a=e(0)$ and $b=e(1)$. Furtherm, a {\em partition} of the interval $[0,1]$ is a finite sequence $(x_i)_{i=0}^{k}$, such that 
$0=x_0 < x_1 < \ldots< x_k=1$.
Thus, the {\em length} of the arc $A=e([0,1])$ within endpoints $a$ and $b$ is defined by
$$\mathcal{L}_a^b(A)= \sup\Bigg\{\sum_{i=1}^k d(e(x_{i-1}),e(x_i)): (x_i)_{i=0}^{k} \text{ is a partition of } [0,1]\Bigg\},$$
where $d$ refers to the standard Euclidean distance in $\R^n$. Note that the length is independent of the choice of the embedding $e$.  

Note that Theorem \ref{thm_Sanders} means that the space $S$ is not an attractor for any iterated function system.
Accordingly, to conclude the proof, we will show that the snake becomes an attractor for the weak IFS $\{f,g_1,\dots,g_m\}$, where the functions $g_i\colon S\lto S$ do project the snake onto finite length intervals, which allow to cover parts of the space. We can even choose the functions $g_1,\dots,g_m$ being contractions, as well as in \cite[Theorem 3.1]{Sanders}. Moreover, the function $f$ has to fill the remaining part of the snake, which has infinite length. Note that it scales down the modulus of points, namely
$$f(r,\alpha)=(\tilde f(r),\alpha),$$ where $\tilde f(r)$ is defined as follows: 

\begin{equation*}
\tilde f(r)=
\begin{cases} 
0 &\mbox{if } r=0; \\ 
\frac{rn(n+1)+2}{(n+2)(n+3)} & \mbox{if } r\in [\frac1{n+1},\frac1n]. 
\end{cases}
\end{equation*} 
In fact, note that whether $r\in [\frac1{n+1},\frac1n]$, then there exists $t\in[0,1]$ such that $r=\frac1{n+1}+t\,(\frac1n -\frac1{n+1})$. Hence, 
$$\tilde f(r)=\frac1{n+3}+t\,\Big(\frac1{n+2}-\frac1{n+3}\Big)= \frac{rn(n+1)+2}{(n+2)(n+3)}.$$
In addition to that, note that $f(O_n\cup I_n) = O_{n+2}\cup I_{n+2}$, for each $n\geq 1$, so the set $f(S)$ covers $S\setminus(O_1\cup I_1\cup O_2\cup I_2)$, and $\bigcup_{i=1}^m g_i(S)$ covers $O_1\cup I_1\cup O_2\cup I_2$. Hence, the snake becomes the attractor of the system $\{f,g_1,\dots,g_m\}$. For completeness, let us show that $f\colon S\lto S$ is a weak contraction. 

Indeed, it becomes straightforward to check the two next statements: 
\begin{enumerate}[(a)]
\item \label{obserwacja1} $r> \tilde f(r)$, for every $r\in(0,1]$;
\item \label{obserwacja2} $\tilde f\colon[0,1]\lto[0,1]$ is strictly decreasing (that is, $\tilde f(r)> \tilde f(p)$, whenever $r>p$). 
\end{enumerate}
Now, let $x=(r_x,\alpha_x)$ and $y=(r_y,\alpha_y)$ be any two fixed but arbitrarily chosen distinct points from $S$. Recall that our goal here is to prove that
\begin{equation}\label{eqExpectedResult}
d(x,y)>d(f(x),f(y)).
\tag{$\star$}
\end{equation}
To deal with, we will distinguish the two following cases:
\newline
\newline
\textbf{Case 1.} $r_x=0$ or $r_y=0$. Without loss of generality, let us suppose $r_y=0$. Then $r_x>0$ and by (\ref{obserwacja1}), we obtain  $d(x,y)=r_x > \tilde f(r_x) = d(f(x),f(y))$.
\newline
\newline
\textbf{Case 2.} $x\in O_n\cup I_n$ and $y\in O_k\cup I_k$, for some integers $n,k\geq 1$. We may assume that $r_x\geq r_y$. Therefore, $n\leq k$. From the cosine formula, we have that
$$d(x,y)=\sqrt{r_x^2 + r_y^2 - 2\, r_x r_y \cos(\alpha_x-\alpha_y)},$$
as well as
$$d(f(x),f(y))=\sqrt{\tilde f(r_x)^2 + \tilde f(r_y)^2 - 2\, \tilde f(r_x) \tilde f(r_y) \cos(\alpha_x-\alpha_y)}.$$
Thus, the difference between the squares of those distances satisfies the next expresion:
\begin{equation}
d(x,y)^2-d(f(x),f(y))^2\geq(r_x-r_y)^2-(\tilde f(r_x) - \tilde f(r_y))^2,
\tag{$\diamond$}\label{diament}
\end{equation}
since $\cos(\alpha_x-\alpha_y)\leq 1$.
Further, 
\begin{enumerate}[(i)]
\item if $r_x=r_y$, then $\alpha_x\neq\alpha_y$, so $\cos(\alpha_x-\alpha_y)< 1$. Consequently, the inequality (\ref{diament}) becomes strict, which implies that $d(x,y)^2-d(f(x),f(y))^2> 0$. This gives (\ref{eqExpectedResult}).
\item If $r_x>r_y$ and $n=k$, then 
$$(\tilde f(r_x) - \tilde f(r_y))^2 = \Bigg(\frac{(r_x-r_y)\, n\, (n+1)}{(n+2)(n+3)}\Bigg)^2<(r_x-r_y)^2,$$
so from (\ref{diament}), we again get (\ref{eqExpectedResult}).
\item Finally, if $r_x>r_y$ and $n<k$, then we get
\begin{align*}
\tilde f(r_x) - \tilde f(r_y) &=\frac1{n+3}+t_x\Big(\frac1{n+2}-\frac1{n+3}\Big)-\frac1{k+3}-t_y\Big(\frac1{k+2}-\frac1{k+3}\Big)\\
&=\frac1{n+3}-\frac1{k+2} + t_x\Big(\frac1{n+2}-\frac1{n+3}\Big) + (1-t_y)\Big(\frac1{k+2}-\frac1{k+3}\Big)\\
&<\frac1{n+3}-\frac1{k+2} + t_x\Big(\frac1{n}-\frac1{n+1}\Big) + (1-t_y)\Big(\frac1{k}-\frac1{k+1}\Big)\\
&\leq\frac1{n+1}-\frac1{k} + t_x\Big(\frac1{n}-\frac1{n+1}\Big) + (1-t_y)\Big(\frac1{k}-\frac1{k+1}\Big)\\
&=r_x-r_y.
\end{align*}
From (\ref{obserwacja2}), both sides of this inequality become positive, so
$$(r_x-r_y)^2 - (\tilde f(r_x) - \tilde f(r_y))^2>0,$$
and from (\ref{diament}), we get (\ref{eqExpectedResult}). 
\end{enumerate}

The previous arguments allow us to affirm that $f$ is a weak contraction, which completes the proof.
\end{proof}

The following counterexample leads to affirm that topological fractal does not imply Banach fractal, in general. 

\begin{ct}
There exists a Peano continuum called the ``shark teeth'' which is not a Banach fractal but it is a topological fractal.
\end{ct}

Recall that by a Peano continuum, we understand a continuous image of the closed unit interval $[0,1]$. The space presented next was firstly constructed in \cite{BanakhNowak}, and studied later in \cite{NS}.

\begin{proof}
Let us consider the following piecewise linear periodic function:
\begin{equation*}
\varphi(t)=
 \begin{cases}
t-n & \text{if }t\in[n,n+\frac{1}{2}] \text{ for some }n\in\Z;  \\
n-t & \text{if }t\in[n-\frac{1}{2},n] \text{ for some }n\in\Z,
\end{cases}
\end{equation*}
whose graph looks like as follows:

\begin{picture}(300,80)(-30,-10)
\put(-20,0){\vector(1,0){300}}
\put(275,-10){$t$}
\put(0,-10){\vector(0,1){70}}
\put(-20,20){\line(1,-1){20}}
\put(0,0){\line(1,1){40}}
\put(40,40){\line(1,-1){40}}
\put(80,0){\line(1,1){40}}
\put(120,40){\line(1,-1){40}}
\put(160,0){\line(1,1){40}}
\put(200,40){\line(1,-1){40}}
\put(240,0){\line(1,1){20}}
\end{picture}

Moreover, for every $n\in\N$, let us define the function
\begin{equation*}
\varphi_n(t)=2^{-n}\, \varphi(2^nt),
\end{equation*}
which is a homothetic copy of the function $\varphi(t)$.

On the other hand, {\em Shark teeth} type spaces were constructed in \cite{BT}. They are parame\-trized through an infinite non-decreasing sequence $(n_k)_{k=1}^\infty$. Thus, let $I=[0,1]\times\{0\}$ be the {\em bone} of the shark teeth, and for every $k\in\N$, let $M_k= \big\{\big(t,\frac{1}k\, \varphi_{n_k}(t)\big):t\in[0,1]\big\}$ be the $kth$ {\em row} of the teeth. The space called shark teeth is given by 
$$M =I\cup \bigcup_{k=1}^\infty M_k.$$
\begin{figure}[h]
\includegraphics[scale=0.65]{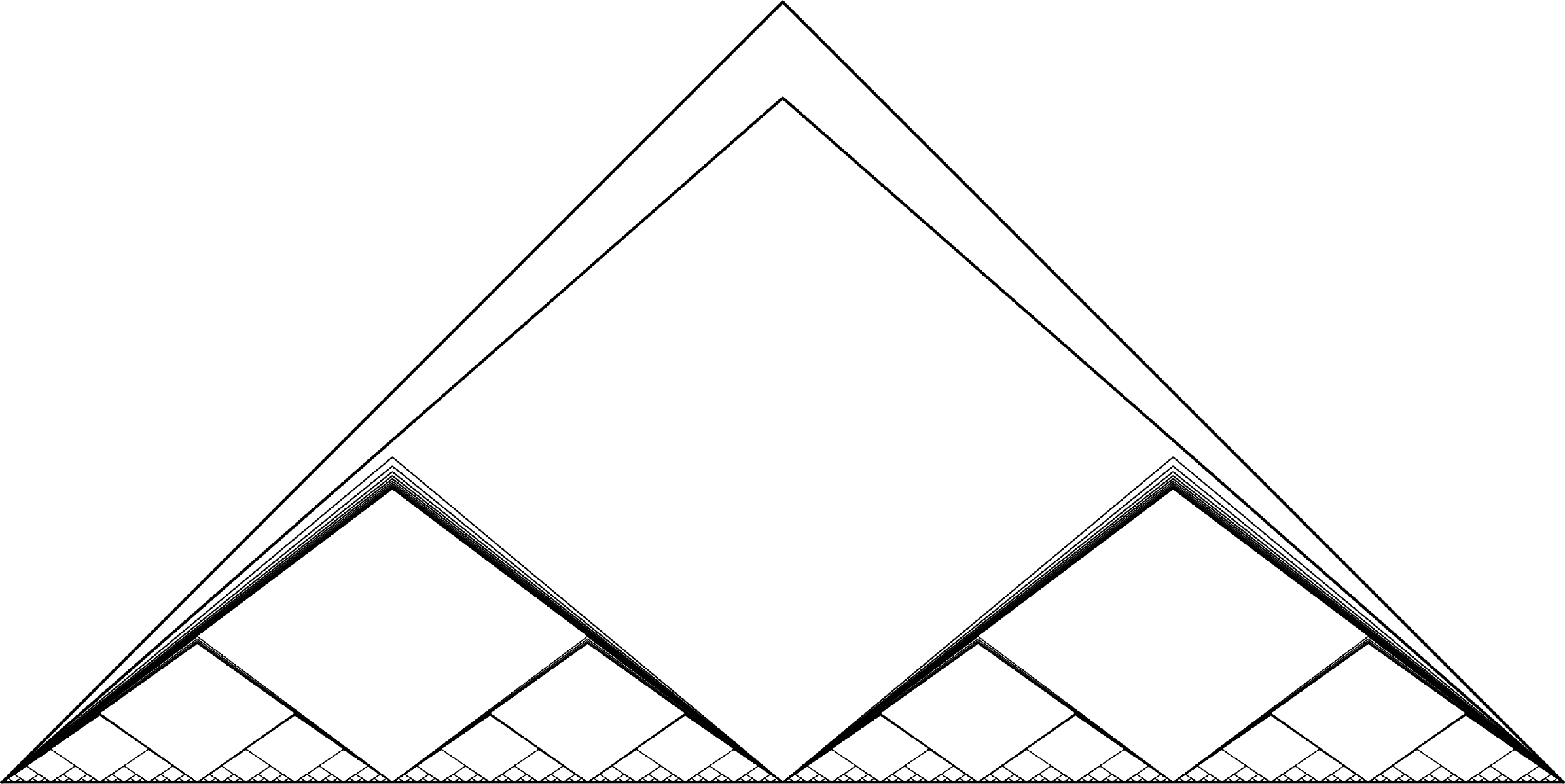}
\caption{The \emph{shark teeth}.}\label{shark}
\end{figure}
In \cite{BanakhNowak}, it was proved that the shark teeth constructed in the plane $\R^2$ through the non-decreasing sequence
\begin{equation*}
n_k=\lfloor \log_2 \log_2 (k+1)\rfloor,\quad k\in\N,
\end{equation*}
where $\lfloor \cdot \rfloor$ denotes the integer part, is not homeomorphic to any IFS-attractor. In other words, this is not a Banach fractal. However, it can be still shown that this is a topological fractal, which was proved in \cite{NS}, though we contribute a shorter proof next.

\begin{prop}\label{prop:1}
Each Peano continuum $P$ within a free arc (namely, a segment $L$ such that $L\cap\overline{P\setminus L}$ only consists of either one or two points) is a topological fractal. 
\end{prop}
 
\begin{proof}
Recall that a Peano continuum $P$ is a continuous image of the unit interval. For the free arc $L$, there exists a continuous injective map $\rho:[0,1]\lto L$, which divides $P$ into three pieces: $P=P_1\, \cup L\, \cup P_2$, where $P_1\cap L=\{\rho(0)\}$, $P_2\cap L=\{\rho(1)\}$, and there also exist continuous functions $\rho_i:[0,1]\lto P_i$, for $i=1,2$.
We may assume that $\rho_i(0)=\rho_i(1)$, for $j\in J\setminus \{0\}$, where $J=\{0,1,2\}$. Further, let us consider the tent map:
\begin{equation*}
\varphi(x)=
 \begin{cases}
x & \text{if }x\in[0,\frac{1}{2}];  \\
1-x & \text{if }x\in[\frac{1}{2},1],
\end{cases}
\end{equation*}
as well as three contractions acting on $[0,1]$, given as follows. For all $i\in J$, let us define
$f_i(x) = \frac{i+2\, \varphi(x)}3$. Note that $\Lip f_i=\frac{2}3<1$, and $\bigcup_{i\in J}f_i([0,1])=\bigcup_{i\in J}[\frac{i}3,\frac{i+1}3]=[0,1]$.  

Next, we explain how to construct a finite family $\F=\{F_i, G_j\colon P\lto P: i\in J, j\in J\setminus \{0\}\}$, which consists of continuous functions for which $P$ is a topological fractal. In fact, for all $i\in J$, let
$$F_i(x)=
 \begin{cases}
\rho\circ f_i \circ \rho^{-1}(x) & \text{for }x\in L;  \\
\rho(\frac{i}3) & \text{for }x\notin L,
\end{cases}
$$
and for $j\in J\setminus \{0\}$, let us consider
$$G_j(x)=
 \begin{cases}
\rho_j\circ  \rho^{-1}(x) & \text{for }x\in L;  \\
\rho_j(0) & \text{for }x\notin L.
\end{cases}
$$
Hence, $\bigcup_{i\in J}F_i(P)=L$, and $G_j(P)=P_j$, for $j\in J\setminus \{0\}$. Moreover, $\Lip\, F_i=\frac{2}3<1$ for all $i\in J$, and it is also satisfied that both $G_1$ and $G_2$ are uniformly continuous (as a continuous function on a compact space). They allow to affirm that, for an arbitrary open cover $\U$ of Peano continuum $P$, and for its Lebesgue number $\varepsilon$, there exists $\delta>0$ such that for every $x,y\in P$, the condition $d(x,y)<\delta$ implies that $d(G_j(x),G_j(y))<\varepsilon$, where $d$ is a metric on $P$. Let us fix a natural number $m$ such that $(\frac{2}3)^m<\min(\delta,\varepsilon)$. Then for every $g_0,\dots,g_m\in\F$, the diameter of the set $g_0\circ\ldots\circ g_m(P)$ is less than $\varepsilon$.

Indeed, for every $i\in J$, and $j,k\in J\setminus \{0\}$, the image sets $F_i\circ G_j(P)$ and $G_j\circ G_k(P)$ are singletons. Thus, if $\{g_1,\dots,g_m\}\cap\{G_1,G_2\}\neq\emptyset$, then the set $g_0\circ\ldots\circ g_m(P)$ has diameter equal to $0$. Otherwise, $\{g_1,\dots,g_m\}\subset\{F_i:i\in J\}$, so $\diam(g_1\circ\ldots\circ g_m(P))\leq(\frac{2}3)^m<\min(\delta,\varepsilon)$. This implies that the diameter of the set $g_0\circ\ldots\circ g_m(P)$ is less than $\varepsilon$, and accordingly, $P$ is a topological fractal. 
\end{proof}

Proposition \ref{prop:1} generalizes the result obtained independently by D. Dumitru, who showed in \cite{Dumitru} that the union of a Peano continuum and a segment, such that their intersection is a singleton, becomes a topological fractal. Hence, the desired result yields as a consequence of former arguments:
\begin{cor}
The shark teeth is a topological fractal.
\end{cor} 
This completes the proof for this counterexample.
\end{proof}

Finally, we state that the remaining connection (weak IFS-attractor $\Rightarrow$ topological fractal) cannot be turned back.
\begin{ct}
There exists a Peano continuum $P$ in $\R^2$ which is not the attractor of any weak IFS but it is a Banach fractal.  
\end{ct}

\begin{proof} 
To provide an appropriate definition for the space $P$, once again let us switch into standard polar coordinates in $\R^2$. Indeed, for $n\geq 1$, let us define $\rho_0=(0,0)$, as well as $\rho_n=(2^{-n},2^{-n})$. For any $n\geq 1$, choose a piecewise linear arc $L_n$ (consisting of finitely many line segments) without self-intersections, which starts at $\rho_0$, ends at $\rho_n$, has a total length equal to $2^n$, and is contained in the set $\big[[0,2^{-n})\times(2^{-n}-2^{-n-2},2^{-n}+2^{-n-2})\big]\cup\{\rho_n\}$. Thus, let us define $P=\bigcup_{i=1}^\infty L_i$.
\begin{figure}[h]
\centering
\includegraphics[width=250px]{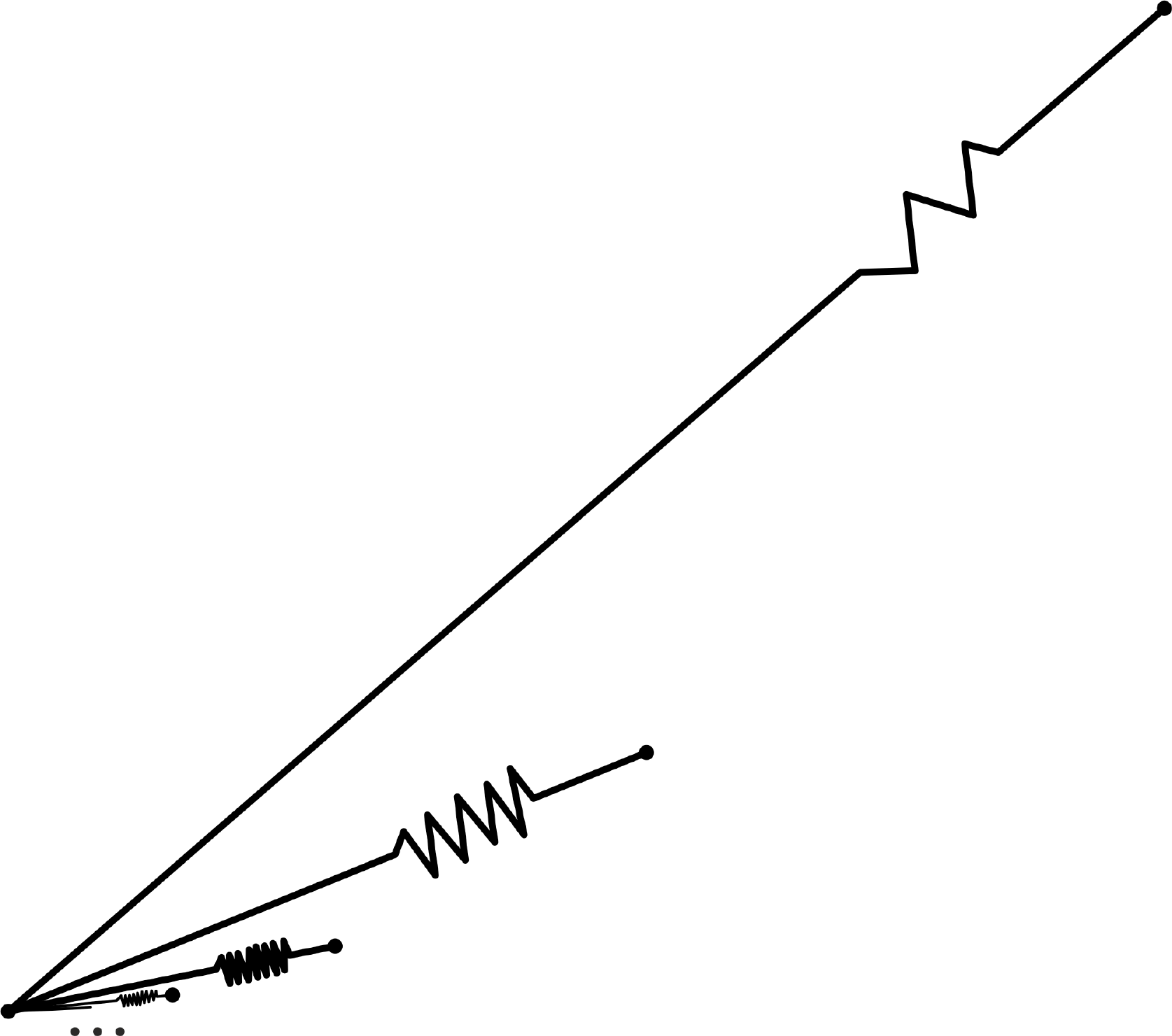}
\caption{The dendrite.}
\end{figure}
In \cite{KN}, it was proved that such space cannot be a weak IFS-attractor, though it is easy to show that it is a Banach fractal. In fact, we can topologically transform the space $P$ in a way such that every arc $L_n$ is a straight line segment whose length is equal to $2^{-n}$.   
\end{proof}

\subsection{Non-topological fractals}\label{sub:2}
So far, we do not know any Peano continuum which is not a topological fractal, though we do know some of such examples in the context of countable compact spaces (also called scattered spaces). 

Let us recall some basic notions related to scattered spaces.
A topological space $X$ is said to be {\em scattered} if every non-empty subspace $Y$ has an isolated point which lies in $Y$. It is well known the fact that any compact metric space is scattered if and only if it is countable. Moreover, every compact scattered space is zero-dimensional, namely, it has a base consisting of clopen sets. \newline
Further, for a scattered space $X$ let
$$X'=\{x\in X \colon x\text{ is an accumulation point of } X\}.$$
be the Cantor-Bendixson derivative of $X$. Inductively we could define
\begin{itemize}
\item $X^{(\alpha+1)}=(X^{(\alpha)})'$;
\item $X^{(\alpha)}=\bigcap_{\beta<\alpha}X^{(\beta)}$, for a limit ordinal $\alpha$.
\end{itemize}
Moreover, the \emph{height} of a scattered space $X$ is given by $$h(X) = \min\{\alpha \colon X^{(\alpha)} \text{ is discrete}\}.$$

\begin{ejem}
There exists countable space which is not a topological fractal.
\end{ejem}

\begin{proof}
In \cite{scattered}, it was proved that every compact scattered metric space with the limit height is not a topological IFS-attractor. Due to that fact, it is enough to construct a space with the limit height, which becomes a standard construction. 

In fact, let us consider a convergent sequence in the real line, whose height is equal to $1$. Thus, we can equate it with the space $\omega +1$, with the order topology. If for each point in this space we take a sequence which converges to such a point, then we obtain a space $\omega^2+1$. If we do the same on that space, then $\omega^3+1$ holds. We will follow that construction, so we can get any space $\omega^n+1$, for every $n\in\N^+$. The height of such spaces is $n$ (a successor ordinal).

Let us construct a space $K\subset[0,1]$, which becomes equivalent to $\omega^\omega +1$ for a limit height. It consists of countable many disjoint blocks which converge to $0$. Their heights are $h_n\lto\omega$ as $n\lto\infty$, so the whole space has a limit height equal to $\omega$.
\begin{figure}[h]
\centering
\includegraphics[scale=0.2]{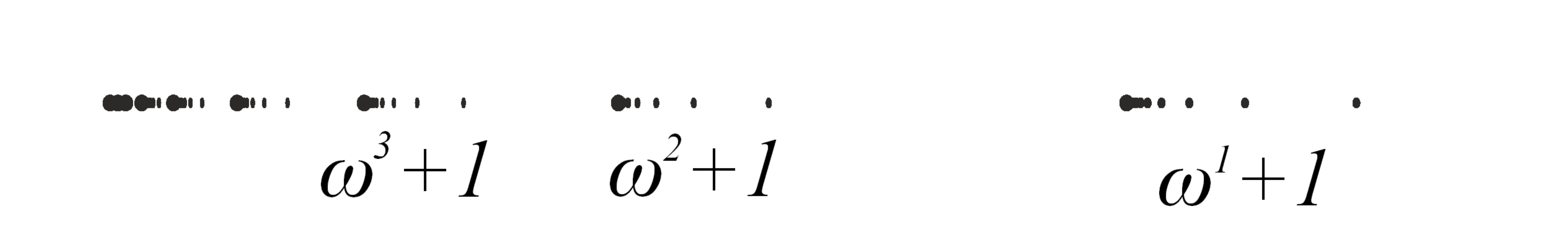}
\caption{An space equivalent to $\omega^\omega +1$.}
\end{figure}
According to \cite[Theorem 3]{scattered}, the space $K$ is not a topological fractal. 
\end{proof}

It turns out that countable compact spaces with the limit height are the only non-topological fractals among zero-dimensional spaces. Moreover, in zero-dimensional sets both notions of Banach and topological fractals are equivalent. This was shown in \cite{BNS}.

\subsection*{Acknowledgements}
The authors would like to thank Taras Banakh for several fruitful discussions and useful hints.



%
%
%


\begin{thebibliography}{99}


\bibitem{MSG12} F.G. Arenas and M.A. Sa\'nchez-Granero, 
\emph{A characterization of self-similar symbolic spaces,}
Mediterr. J. Math. 9 (4) (2012) 709--728.

\bibitem{BKNNS} T. Banakh, W. Kubi\'s, N. Novosad, M. Nowak, and F. Strobin, 
\emph{Contractive function systems, their attractors and metrization,}
Topol. Methods Nonlinear Anal., to appear. arXiv:1405.6289v1 [math.GN].

\bibitem{BanakhNowak} T. Banakh and M. Nowak, 
\emph{A 1-dimensional Peano continuum which is not an IFS attractor,}
Proc. Amer. Math. Soc. 141 (3) (2013) 931--935.

\bibitem{BNS} T. Banakh, M. Nowak, and F. Strobin, 
\emph{Detecting topological and Banach fractals among zero-dimensional spaces,}
Topology Appl. 196 A (2015) 22--30.

\bibitem{BT} T. Banakh and M. Tuncali, 
\emph{Controlled Hahn-Mazurkiewicz Theorem and some new dimension functions of Peano continua,}
Topology Appl. 154 (7) (2007) 1286--1297.

\bibitem{Barnsley} M.~Barnsley, {\em Fractals everywhere}, Academic Press, Boston, 1988.

\bibitem{Dumitru} D. Dymitru, 
\emph{Attractors of topological iterated function system,}
Annals of Spiru Haret University: Mathematics-Informatics series 8 (2) (2012) 11--16.

\bibitem{Edelstein} M. Edelstein, 
\emph{On fixed and periodic points under contractive mappings,}
J. Lond. Math. Soc. s1-37 (1) (1962) 74--79.

\bibitem{DIM1} M. Fern\'andez-Mart\'{\i}nez and M.A. S\'anchez-Granero,
\emph{Fractal dimension for fractal structures,}
Topology Appl. 163 (2014) 93--111.

\bibitem{DIM3} M. Fern\'andez-Mart\'{\i}nez and M.A. S\'anchez-Granero,
\emph{Fractal dimension for fractal structures: A Hausdorff approach,}
Topology Appl. 159 (7) (2012) 1825--1837.

\bibitem{NDY15} M. Fern\'andez-Mart\'{\i}nez, M.A. S\'anchez-Granero, J.E. Trinidad Segovia, and J.A. Vera, 
\emph{A new topological indicator for chaos in mechanical systems,}
Nonlinear Dynamics, June (2015) 1--13. doi: 10.1007/s11071-015-2207-x

\bibitem{KN} M. Kulczycki and M. Nowak, 
\emph{A class of continua that are not attractors of any IFS,}
Cent. Eur. J. Math. 10 (6) (2012) 2073--2076.

\bibitem{MAN77} B.B. Mandelbrot, 
\emph{Fractals: Form, Chance and Dimension,}
W.H. Freeman \& Company, San Francisco, 1977.

\bibitem{MAN82} B.B. Mandelbrot, 
\emph{The Fractal Geometry of Nature,}
W.H. Freeman \& Company, New York, 1982.

\bibitem{mihail} A.~Mihail {\em A topological version of iterated function systems}, An. \c Stiint.Univ. Al. I. Cuza, Ia\c si, (S.N.), Matematic\u a, LVIII, (2012), f.1, 105--120. 

\bibitem{scattered} M. Nowak, 
\emph{Topological classification of scattered IFS-attractors,}
Topology Appl. 160 (14) (2013) 1889--1901.

\bibitem{NS} M. Nowak and T. Szarek, 
\emph{The shark teeth is a topological IFS-attractor,}  
Siberian Mathematical Journal 55 (2) (2014) 296--300.

\bibitem{MOR46} P.A.P. Moran, 
\emph{Additive functions of intervals and Hausdorff measure,}
Proc. Camb. Phil. Soc. 42 (1) (1946) 15--23.

\bibitem{ROB95} C. Robinson,
\emph{Dynamical Systems: Stability, Symbolic Dynamics and Chaos}, CRC Press Inc., Boca Raton (FL), 1995.

\bibitem{Sanders} M.J. Sanders, 
\emph{Non-attractors of iterated function systems,}
Texas Project NexT e-Journal 1 (2003) 1--9.

\bibitem{TAR15} M. Tarnopolski, 
\emph{Correlations between Hurst Exponent and maximal Lyapunov Exponent for some low-dimensional discrete conservative dynamical systems,}
preprint, arXiv:1501.03766v1 [nlin.CD].

\end{thebibliography}
\end{document}